\documentclass{amsart}

\usepackage{graphicx}
\usepackage{amsmath}
\usepackage{amsfonts}
\usepackage{amsthm}
\usepackage{amssymb,latexsym,hyperref}

\newtheorem{thm}{Theorem}[section]
\newtheorem{lem}[thm]{Lemma}
\newtheorem{claim}[thm]{Claim}
\newtheorem{defn}[thm]{Definition}
\newtheorem{ques}[thm]{Question}
\newtheorem{fact}[thm]{Fact}
\newtheorem{cor}[thm]{Corollary}

\def\res{\upharpoonright}
\def\name{\mathring}

\def\dom{\text{dom}}

\def\bb{\mathbb}
\def\cal{\mathcal}
\def\frak{\mathfrak}
\def\e{\varepsilon}

\def\Cohen{\textsf{Cohen}}
\def\Random{\textsf{Random}}
\def\cf{\textsf{cf}}
\def\dom{\textsf{dom}}
\def\rng{\textsf{range}}
\def\supp{\textsf{supp}}

\def\cont{\frak{c}}

\def\otp{\textsf{otp}}

\def\add{\textsf{add}}

\begin{document}

\nocite{*}

\title{Supersaturated ideals}

\author[Kumar]{Ashutosh Kumar}
\address[Kumar]{Department of Mathematics and Statistics, Indian Institute of Technology Kanpur, Kanpur 208016, UP, India.}
\email{\href{krashu@iitk.ac.in}{krashu@iitk.ac.in}}
\urladdr{\url{https://home.iitk.ac.in/~krashu/}}

\author[Raghavan]{Dilip Raghavan}
\address[Raghavan]{Department of Mathematics, National University of Singapore, 10 Lower Kent Ridge Road, Singapore 119076.}
\email{\href{mailto:dilip.raghavan@protonmail.com}{dilip.raghavan@protonmail.com}}
\urladdr{\url{https://dilip-raghavan.github.io/}}
\thanks{Both authors were partially supported by Singapore Ministry of Education's research grant number MOE2017-T2-2-125}

\date{}

\begin{abstract}

A $\sigma$-ideal $\cal{I}$ on a set $X$ is supersaturated if for every family $\cal{F}$ of $\cal{I}$-positive sets with $|\cal{F}| < \add(\cal{I})$,  there exists a countable set that meets every set in $\cal{F}$.  We show that many well-known ccc forcings preserve supersaturation.  We also show that the existence of supersaturated ideals is independent of ZFC plus ``There exists an $\omega_1$-saturated $\sigma$-ideal".

\end{abstract}

\maketitle

\section{Introduction}

Saturation properties of ideals are ubiquitous in modern set theory and there is a considerable body of work (for example, see \cite{Kan, Kun, Lav,  Sol}) on the study of a large number of such properties.  Supersaturation is a strengthening of $\omega_1$-saturation defined as follows.

\begin{defn}
\label{d11}
Let $\cal{I}$ be a $\sigma$-ideal on $X$. We say that $\cal{I}$ is $\theta$-supersaturated iff for every $\cal{A} \subseteq \cal{I}^+$, if $|\cal{A}| < \add(\cal{I})$, then there exists $W \in [X]^{< \theta}$ such that for every $A \in \cal{A}$, $A \cap W \neq \emptyset$. $\cal{I}$ is supersaturated iff it is $\omega_1$-supersaturated.
\end{defn}

Though closely related to some of the works of Fremlin,  supersaturated ideals were formally introduced in \cite{KR}  where it was shown that if $\kappa \leq \cont$ admits a normal supersaturated ideal then the order dimension of the Turing degrees is at least $\kappa$. An earlier motivation for investigating these ideals comes from the following question of Fremlin -- See Problem EG(h) in \cite{Frem}. 

\begin{ques}[Fremlin]
\label{q12}
Suppose $\kappa$ is real valued measurable and $m:\cal{P}(\kappa) \to [0, 1]$ is a witnessing normal measure.  Let $\cal{F}$ be a family of subsets of $\kappa$ such that $|\cal{F}| < \kappa$ and for every $A \in \cal{F}$, $m(A) > 0$. Must there exist a countable $N \subseteq \kappa$ such that for every $A \in \cal{F}$, $N \cap A \neq \emptyset$?
\end{ques}

So Question \ref{q12} is asking if the null ideal of every normal witnessing measure on a real valued measurable cardinal must be supersaturated.  It is easy to see that every supersaturated ideal is $\omega_1$-saturated.  One of the standard ways of obtaining $\omega_1$-saturated ideals on cardinals below the continuum is to start with a measurable cardinal $\kappa$ and a witnessing normal prime ideal $\cal{I}$ on $\kappa$, and force with a ccc forcing $\bb{P}$ that adds $\geq \kappa$ reals. Let $\cal{J}$ be the ideal generated by $\cal{I}$ in $V^{\bb{P}}$. Then $\cal{J}$ is always an $\omega_1$-saturated normal ideal on $\kappa \leq \cont$.  But whether or not $\cal{J}$ is supersaturated will depend on the choice of $\bb{P}$.  This motivates the notion of supersaturation preserving forcings (Definition \ref{d21}).  In Section \ref{s2},  we show that a large class of ccc forcings for adding new reals are supersaturation preserving. In particular,  the following holds.

\begin{thm} 
\label{t13}
Let $\Random_{\lambda}$ denote the forcing for adding $\lambda$ random reals.

\begin{itemize} 
\item[(1)] Every $\sigma$-linked forcing is supersaturation preserving. 

\item[(2)] $\Random_{\lambda}$ is supersaturation preserving for every $\lambda$.

\end{itemize}
\end{thm}

The question of whether every $\omega_1$-saturated ideal must be supersaturated was raised in \cite{KR}.  Our main result shows that this is independent.

\begin{thm}
\label{t14}
Each of the following is consistent.

\begin{itemize}

\item[(1)] There is an $\omega_1$-saturated ideal on a cardinal below the continuum and there are no supersaturated ideals.

\item[(2)] There is an $\omega_1$-saturated ideal on a cardinal below the continuum and every $\omega_1$-saturated ideal is supersaturated.

\end{itemize}
\end{thm}

\textbf{Notation}:  Let $\cal{I}$ be an ideal on $X$.  Define $\cal{I}^+ = \cal{P}(X) \setminus \cal{I}$.  $\add(\cal{I})$ denotes the least cardinality of a subfamily of $\cal{I}$ whose union is in $\cal{I}^+$.  For $A \subseteq X$, define $\cal{I} \res A = \{Y \subseteq X: Y \cap A \in \cal{I}\}$. For a set of ordinals $X$, $\otp(X)$ denotes the order type of $X$.  An ordinal $\delta$ is indecomposable iff for every $X \subseteq \delta$, either $\otp(X) = \delta$ or $\otp(\delta \setminus X) = \delta$. If $\bb{P}$, $\bb{Q}$ are forcing notions, we write $\bb{P} \lessdot \bb{Q}$ iff $\bb{P} \subseteq \bb{Q}$ and every maximal antichain in $\bb{P}$ is also a maximal antichain in $\bb{Q}$.  $\Cohen_{\lambda}$ denotes the forcing for adding $\lambda$ Cohen reals.  $\Random_{\lambda}$ is the measure algebra on $2^{\lambda}$ equipped with the usual product measure denoted by $\mu_{\lambda}$. If $\lambda$ is clear from the context, then we drop it and just write $\mu$.

\section{CCC forcings and supersaturation}
\label{s2}

\begin{defn}
\label{d21}
A forcing $\bb{P}$ is $\kappa$-ssp (ssp = supersaturation preserving) iff for every normal supersaturated ideal $\cal{I}$ on $\kappa$, $V^{\bb{P}} \models $ ``the ideal generated by $\cal{I}$ is supersaturated".  $\bb{P}$ is ssp iff it is $\kappa$-ssp for every $\kappa$.

\end{defn}

In \cite{KR}, the following forcings were shown to be $\kappa$-ssp for every $\kappa$.

\begin{itemize}
\item[(a)] $\Cohen_{\lambda}$ for any $\lambda$.
\item[(b)] Any finite support iteration of ccc forcings of size $< \kappa$. 
\end{itemize}

It was also shown that $\Random_{\lambda}$ is $\kappa$-ssp for any measurable $\kappa$. The next theorem improves this to all $\kappa$. 

\begin{thm}
\label{t22}
$\Random_{\lambda}$ is $\kappa$-ssp for every $\kappa$ and $\lambda$. 
\end{thm}

\begin{proof}

Fix a normal supersaturated ideal $\cal{I}$ on $\kappa$. Put $\bb{B} = \Random_{\lambda}$ and let $\cal{J}$ be the ideal generated by $\cal{I}$ in $V^{\bb{B}}$. Suppose $\theta < \kappa$ and $\Vdash_{\bb{B}}  \langle \name{A}_i: i < \theta \rangle$ is a sequence of $\cal{J}$-positive sets. It suffices to find $B \in [\kappa]^{\aleph_0}$ such that $\Vdash_{\bb{B}} (\forall i < \theta)(\name{A}_i \cap B \neq \emptyset)$. \\

For $i < \theta$ and $\alpha < \kappa$, put $p_{i, \alpha} = [[\alpha \in \name{A}_i]]_{\bb{B}}$. So each $p_{i, \alpha}$ is a Baire subset of $2^{\lambda}$. Put $T_i = \{\alpha < \kappa: p_{i, \alpha} \neq 0_{\bb{B}}\}$.

\begin{claim}
\label{c23}
For each $p \in \bb{B} \setminus \{0_{\bb{B}}\}$, $\{\alpha \in T_i: p_{i, \alpha} \cap p \neq 0_{\bb{B}} \} \in \cal{I}^+$.
\end{claim}
\begin{proof} Put $X_p = \{\alpha \in T_i: p_{i, \alpha} \cap p \neq 0_{\bb{B}} \}$ and suppose $X_p \in \cal{I}$. Since the empty condition forces that $\name{A}_i \in \cal{J}^+$, it follows that for every $X \in \cal{I}$, $\{ p_{i, \alpha} : \alpha \in T_i \setminus X\}$ is predense in $\bb{B}$. But every condition in $\{ p_{i, \alpha} : \alpha \in T_i \setminus X_p\}$ is incompatible with $p$ which is impossible.
\end{proof}

For a finite partial function $f$ from $\lambda$ to $2$, define $[f] = \{x \in 2^{\lambda}: x \res \dom(f) = f\}$. For a clopen $K \subseteq 2^{\lambda}$, define $\supp(K)$ to be the smallest finite set $S \subseteq \lambda$ such that $(\forall x, y \in 2^{\lambda})(x \res S = y \res S \implies (x \in K \iff y \in K))$. If $\supp(K) = S$, then there there is finite list $\{f_{K, n}: n < n_{\star}\}$ where $f_{K, n}$'s are pairwise distinct functions from $S$ to $2$ and $K = \bigsqcup_{n < n_{\star}} [f_{K, n}]$.

\begin{defn}
\label{d24}
Suppose $\cal{C}$ is a family of clopen sets in $2^{\lambda}$. We say that $\cal{C}$ is a strong $\Delta$-system of width $(n_{\star}, N_{\star})$ iff $n_{\star}, N_{\star} < \omega$ and the following hold.

\begin{itemize}

\item[(a)] $\langle \supp(K): K \in \cal{C} \rangle$ is a $\Delta$-system with root $R$.

\item[(b)] For every $K \in \cal{C}$, $|\supp(K) \setminus R| = n_{\star}$.

\item[(c)] For every $K \in \cal{C}$, $K = \bigsqcup_{n < N_{\star}} [f_{K, n}]$ where each $f_{K, n}: \supp(K) \to 2$ and $f_{K, n}$'s are pairwise distinct.

\item[(d)] For every $K_1, K_2 \in \cal{C}$ and $n < N_{\star}$, 

\subitem (i) $f_{K_1, n} \res R = f_{K_2, n} \res R$ and

\subitem (ii) if for $m \in \{1, 2\}$, $\{\xi^m_j : j < |R| + n_{\star}\}$ lists $\supp(K_m)$ in increasing order, then $f_{K_1, j}(\xi^1_j) = f_{K_2, j}(\xi^2_j)$ for every $j < |R| + n_{\star}$.

\end{itemize}
 
\end{defn}

\begin{lem}
\label{l25}
Suppose $p \subseteq 2^{\lambda}$ is Baire and $\cal{C}$ is an infinite strong $\Delta$-system of clopen sets in $2^{\lambda}$ of width $(n_{\star}, N_{\star})$. Let $\e > 0$ and assume that for infinitely many $K \in \cal{C}$, $\mu(p \cap K) \geq \e$. Then for all but finitely many $K \in \cal{C}$, $\mu(p \cap K) \geq \e/2$.
\end{lem}

\begin{proof}
Let $R$ be the root of $\langle \supp(K): K \in \cal{C}\rangle$.  For each $K \in \cal{C}$, fix $\langle f_{K, n}: n < N_{\star} \rangle$ such that $K = \bigsqcup_{n < N_{\star}} [f_{K, n}]$.  First suppose that $p$ is clopen. Let $\cal{C}_p = \{K \in \cal{C}: (\supp(K) \setminus R) \cap \supp(p) = \emptyset\} $. Then $\cal{C} \setminus \cal{C}_p$ is finite and for each $K \in \cal{C}_p$, $$\mu(p \cap K) = \sum_{n < N_{\star}} \mu(p \cap [f_{K, n}]) = 2^{- n_{\star}} \sum_{n < N_{\star}} \mu(p \cap [f_{K, n} \res R])$$

which does not depend on $K \in \cal{C}_p$. It follows that the result holds if $p$ is clopen. The general case follows by applying the previous case to a clopen $q \subseteq 2^{\lambda}$ satisfying $\mu(p \Delta q) < \e/2$. \end{proof}

For each $\alpha \in T_i$, fix $S_{i, \alpha} \in [\lambda]^{\aleph_0}$ such that $p_{i, \alpha}$ is supported in $S_{i, \alpha}$.  For every $i < \theta$, $\alpha \in T_i$ and $\e > 0$ rational, choose a clopen set $K_{i, \alpha, \e} \subseteq 2^{\lambda}$ with $\supp(K_{i, \alpha, \e}) \subseteq S_{i, \alpha}$ such that $$\frac{\mu(p_{i, \alpha} \Delta K_{i, \alpha, \e})}{\mu(K_{i, \alpha, \e})} < \e $$

\begin{claim}
\label{c26}
For each $i < \theta$ and $\e > 0$ rational, we can find $\cal{F}_{i, \e} \subseteq \cal{I}^+$ and $\langle (n_{i, \e, Y}, N_{i, \e, Y}): Y \in \cal{F}_{i, \e} \rangle$ such that the following hold.

\begin{itemize}

\item[(1)] $\cal{F}_{i, \e}$ is a countable family of pairwise disjoint sets and $T_i \setminus \bigcup \cal{F}_{i, \e} \in \cal{I}$. 

\item[(2)] For each $Y \in \cal{F}_{i, \e}$, $\{ K_{i, \alpha, \e}: \alpha \in Y\}$ is a strong $\Delta$-system of width $(n_{i, \e, Y}, N_{i, \e, Y})$.

\end{itemize} 
\end{claim} 

\begin{proof}
Fix $i < \theta$ and $\e > 0$ rational. To simplify notation, we write $K_{\alpha}$ instead of $K_{i, \alpha, \e}$. It suffices to show that for every $\cal{I}$-positive $X \subseteq T_i$, there exists $Y \subseteq X$ such that $Y \in \cal{I}^+$ and there exist $(n_Y, N_Y)$ such that $\{ K_{\alpha}: \alpha \in Y\}$ is a strong $\Delta$-system of width $(n_Y, N_Y)$. Since then we can take $\cal{F}_{i, \e}$ to be a maximal disjoint family of such $Y$'s. 

Fix a club $E \subseteq \kappa$ such that for every $\gamma \in E$ and $\alpha \in T_i \cap \gamma$, $\max(\supp(K_{\alpha})) < \gamma$. Suppose $X \subseteq T_i \cap E$ and $X \in \cal{I}^+$. Since $\cal{I}$ is normal and the map $\alpha \mapsto \max(\supp(K_{\alpha} \cap \alpha))$ is regressive on $X$, we can find $R \subseteq \kappa$ finite and $Y_1 \subseteq X$ such that $Y_1 \in \cal{I}^+$, $(\forall \alpha \in Y_1)(\supp(K_{\alpha}) \cap \alpha = R)$ and $|\supp(K_{\alpha}) \setminus R| = n_{\star}$ does not depend on $\alpha \in Y_1$. It also follows that $\langle \supp(K_{\alpha}): \alpha \in Y_1 \rangle$ forms a $\Delta$-system with root $R$. For each $\alpha \in Y_1$, let $K_{\alpha} = \bigsqcup_{n < N_{\alpha}} [f_{\alpha, n}]$ where each $f_{\alpha, n}:\supp(K_{\alpha}) \to 2$. Choose $Y_2 \subseteq Y_1$ such that $Y_2 \in \cal{I}^+$ and $N_{\alpha} = N_{\star}$ does not depend on $\alpha \in Y_2$. Finally, choose $Y \subseteq Y_2$ such that $Y \in \cal{I}^+$ and $\{ K_{\alpha}: \alpha \in Y\}$ is a strong $\Delta$-system of width $(n_{\star}, N_{\star})$. \end{proof}
 
Since $\cal{I}$ is supersaturated, we can choose $B \in [\kappa]^{\aleph_0}$ such that for every $i < \theta$, $\e > 0$ rational and $Y \in \cal{F}_{i, \e}$,  we have $|B \cap Y| = \aleph_0$.  It suffices to show that for each $i < \theta$, $\{p_{i, \alpha}: \alpha \in B\}$ is predense in $\bb{B}$.  \\

Suppose not. Fix $i < \theta$ and $p \subseteq 2^{\lambda}$ Baire such that $\mu(p) > 0$ and for every $\alpha \in B$, $\mu(p_{i, \alpha} \cap p) = 0$. Let $X = \{\alpha \in T_i: \mu(p_{i, \alpha} \cap p) > 0\}$. By Claim \ref{c23}, $X \in \cal{I}^+$. Using the argument in the proof of Claim \ref{c26}, we can choose $\e > 0$ rational, $X_{\star} \subseteq X$ and $n_{\star}, N_{\star} < \omega$ such that

\begin{itemize}

\item[(a)] $X_{\star} \in \cal{I}^+$ and for each $\alpha \in X_{\star}$, $\mu(p_{i, \alpha} \cap p) \geq 4\e$.

\item[(b)] $\{ K_{i, \alpha, \e}: \alpha \in X_{\star}\}$ is a strong $\Delta$-system of width $(n_{\star}, N_{\star})$.

\end{itemize}

Choose $Y \in \cal{F}_{i, \e}$ such that $Y \cap X_{\star} \in \cal{I}^+$. Since $|Y \cap X_{\star}| \geq \aleph_0$ and $|Y \cap B| = \aleph_0$, by Lemma \ref{l25}, we can choose $\alpha \in Y \cap B$ such that $\mu(p \cap K_{i, \alpha, \e}) \geq 2\e$. But since $\mu(p_{i, \alpha} \Delta K_{i, \alpha, \e}) \leq \e \mu(K_{i, \alpha, \e}) \leq \e$, it follows that $\mu(p \cap p_{i, \alpha}) \geq \e > 0$: Contradiction. This completes the proof of Theorem \ref{t22}. \end{proof}

\begin{thm}
\label{t27}
Every $\sigma$-linked forcing is $\kappa$-ssp for every $\kappa$.
\end{thm}

\begin{proof}

Let $\cal{I}$ be a normal supersaturated ideal on $\kappa$. Suppose $\bb{P}$ is a $\sigma$-linked forcing and $\cal{J}$ is the ideal generated by $\cal{I}$ in $V^{\bb{P}}$. Fix $\theta < \kappa$ and WLOG, assume that the trivial condition forces that $\langle \name{A}_i : i < \theta \rangle$ is a sequence of $\cal{J}$-positive sets. It suffices to construct $X \in [\kappa]^{\aleph_0}$ such that $\Vdash_{\bb{P}} (\forall i < \theta)(X \cap \name{A}_i \neq \emptyset)$. \\

Since $\bb{P}$ is $\sigma$-linked, we can write $\bb{P} = \bigcup \{L_n: n < \omega\}$ where each $L_n \subseteq \bb{P}$ has pairwise compatible members. For each $i < \theta$ and $n < \omega$, define

$$B_{i, n} = \{\alpha < \kappa: (\exists p \in L_n)(p \Vdash \alpha \in \name{A}_i)\}$$

\begin{claim}
\label{c28}
$W_i = \bigcup \{L_n: n < \omega,  B_{i, n} \in \cal{I}^+ \}$ is dense in $\bb{P}$.
\end{claim}

\begin{proof} Suppose not and fix $p \in \bb{P}$ such that no extension of $p$ lies in $W_i$. Put $C = \{\alpha < \kappa: (\exists q \leq p)(q \Vdash \alpha \in \name{A}_i)\}$. Since no extension of $p$ lies in $W_i$, it follows that $C \subseteq \bigcup \{B_{i, n}: n < \omega, B_{i, n} \in \cal{I}\}$ and hence $C \in \cal{I}$. It now follows that $p \Vdash \name{A}_i \in \cal{J}$ which is impossible. \end{proof}

Since $\cal{I}$ is supersaturated, we can find a countable $X \subseteq \kappa$ such that for every $i < \theta$ and $n < \omega$, if $B_{i, n} \in \cal{I}^+$, then $X \cap B_{i, n} \neq \emptyset$. We claim that $\Vdash (\forall i < \theta)(X \cap \name{A}_i \neq \emptyset)$. Suppose not and fix $p \in \bb{P}$ and $i < \theta$ such that $p \Vdash X \cap \name{A}_i = \emptyset$. Using Claim \ref{c28}, choose $n < \omega$ and $p' \leq p$ such that $p' \in L_n$ and $B_{i, n} \in \cal{I}^+$. Choose $\alpha \in B_{i, n} \cap X$ and $q \in L_n$ such that $q \Vdash \alpha \in \name{A}_i$. Since $L_n$ is linked, we can find a common extension $r \in \bb{P}$ of $p', q$. But $r \Vdash \alpha \in X \cap \name{A}_i$: Contradiction.
\end{proof}

\begin{cor}
\label{c29}
Each of the following forcings is ssp: Cohen, random, Amoeba,  Hechler,  Eventually different real forcing.
\end{cor}

We do not know if we can improve Theorem \ref{t27} to the class of $\sigma$-finite-cc forcings. For example, one can ask the following.

\begin{ques}
\label{q210}
Suppose $\bb{B}$ is a boolean algebra and $m:\bb{B} \to [0, 1]$ is a strictly positive finitely additive measure on $\bb{B}$. Must $\bb{B}$ be supersaturation preserving?
\end{ques}

The next two facts are well known.

\begin{fact}
\label{f211}
Suppose $\bb{P}$ is a separative $\sigma$-linked forcing. Then $|\bb{P}| \leq \cont$.
\end{fact}

\begin{fact}
\label{f212}
Let $\langle (\bb{P}_{\xi}, \name{\bb{Q}}_{\xi}): \xi < \lambda \rangle$ be a finite support iteration with limit $\bb{P}_{\lambda}$ where for every $\xi < \lambda$, $V^{\bb{P}_{\xi}} \models \name{\bb{Q}}_{\xi}$ is $\sigma$-linked. Assume $\lambda < \cont^+$. Then $\bb{P}_{\lambda}$ is also $\sigma$-linked.
\end{fact}

\begin{thm}
\label{c213}
Let $\cal{I}$ be a normal supersaturated ideal on $\kappa$ and let $\lambda \leq \kappa^+$. Suppose $\langle (\bb{P}_{\xi}, \name{\bb{Q}}_{\xi}): \xi < \lambda \rangle$ is a finite support iteration with limit $\bb{P}_{\lambda}$ where for every $\xi < \lambda$, $V^{\bb{P}_{\xi}} \models \name{\bb{Q}}_{\xi}$ is $\sigma$-linked. Let $\cal{J}$ be the ideal generated by $\cal{I}$ in $V^{\bb{P}_{\lambda}}$. Then $\cal{J}$ is supersaturated.
\end{thm} 

\begin{proof}

By induction on $\lambda$. First suppose $\kappa \leq \cont$. If $\lambda < \kappa^+$, then by Fact \ref{f212}, $\bb{P}_{\lambda}$ is $\sigma$-linked and the claim holds by Theorem  \ref{t27}. So assume $\lambda = \kappa^+$ and fix any $\bb{P}_{\lambda}$-generic filter $G_{\lambda}$ over $V$. Let $\langle A_i : i < \theta \rangle$ be a sequence of $\cal{J}$-positive sets in $V[G_{\lambda}]$ where $\theta < \kappa$. Since $\bb{P}_{\lambda}$ is a finite support iteration of ccc forcings, there exists $\eta < \lambda = \kappa^+$ such that $\langle A_i : i < \theta \rangle \in V[G_{\eta}]$ where $G_{\eta} = \bb{P}_{\eta} \cap G_{\lambda}$. Note that each $A_i$ is $\cal{J}_{\eta}$-positive where $\cal{J}_{\eta}$ is the ideal generated by $\cal{I}$ in $V[G_{\eta}]$. By inductive hypothesis, there is a countable set that meets $A_i$ for every $i < \theta$. Hence $\cal{J}$ is supersaturated. 

Next assume $\kappa > \cont$. Then $\kappa$ is measurable and $\cal{I}$ is a normal prime ideal on $\kappa$. First suppose $\lambda \leq \kappa$. By Fact \ref{f211}, $|\bb{P}_{\xi}| \leq |\xi \cdot \cont| < \kappa$ for every $\xi < \kappa$. Hence by Theorem 4.9 in \cite{KR}, it follows that $\cal{J}$ is supersaturated. Next suppose $\kappa < \lambda \leq \kappa^+$. Note that $V^{\bb{P}_{\kappa}} \models \cont \geq \kappa$ since Cohen reals are added at each stage of cofinality $\omega$. So we can work in $V^{\bb{P}_{\kappa}}$ and repeat the argument for the case $\kappa \leq \cont$. 
\end{proof}

It is now natural to ask the following.

\begin{ques} [\cite{KR}]
\label{q214}
Suppose $\kappa$ is measurable. Is every ccc forcing $\kappa$-ssp?
\end{ques}

In Section \ref{s4}, we'll show that the answer is negative. We end this section with the following weaker positive result.

\begin{thm}
\label{t215}
Suppose $\kappa$ is measurable and $\cal{I}$ is a normal prime ideal on $\kappa$. Let $\bb{B}$ be a ccc complete boolean algebra. Then $V^{\bb{B}} \models$ ``the ideal generated by $\cal{I}$ is $\omega_2$-supersaturated."
\end{thm}

\begin{proof} It suffices to show that the following holds in $V^{\bb{B}}$: For every $\cal{A} \subseteq \cal{J}^+$, if $|\cal{A}| < \kappa$, then there exists $X \in [\kappa]^{\aleph_1}$ such that $X$ meets every member of $\cal{A}$. \\

Suppose $\theta < \kappa$ and $\Vdash_{\bb{B}} \{\name{A}_i : i < \theta\} \subseteq \cal{J}^+$. Choose $Y \subseteq \kappa$ of $\cal{I}$-measure one such that for every $i < \theta$ and $\alpha \in Y$, $p_{i, \alpha} = [[\alpha \in \name{A}_i]] > 0_{\bb{B}}$.  Using the inaccessibility of $\kappa$, the following claim is easy to check.

\begin{claim}
\label{c216}
There exists $\langle \bb{B}_{\alpha}: \alpha < \kappa \rangle$ such that the following hold.

\begin{itemize}

\item[(i)] $\bb{B}_{\alpha} \lessdot \bb{B}$ and $|\bb{B}_{\alpha}| < \kappa$.

\item[(ii)] $\bb{B}_{\alpha}$'s are increasing and continuous at $\alpha$ when $\cf(\alpha) > \aleph_0$.

\item[(iii)] $\{p_{i, \beta}: \beta < \alpha, i < \theta\} \subseteq \bb{B}_{\alpha}$.
 
\end{itemize}
\end{claim}

Let $\pi_{\alpha}: \bb{B} \to \bb{B}_{\alpha}$ be a projection map witnessing $\bb{B}_{\alpha} \lessdot \bb{B}$. Choose $f: \kappa \to \kappa$ such that for every $i < \theta$ and $\alpha < \kappa$,  we have $\alpha < f(\alpha)$ and $p_{i, \alpha} \in \bb{B}_{f(\alpha)}$. Choose $Y_1 \subseteq Y$ of measure one and $\alpha_{\star} < \kappa$ such that for every $i < \theta$, $\pi_{\alpha}(p_{i, \alpha}) = p_{i, \star} \in \bb{B}_{\alpha_{\star}}$ does not depend on $\alpha \in Y_1$ and $\rng(f \res \alpha) \subseteq \alpha$ for every $\alpha \in Y_1$. Note that $p_{i, \star} = 1_{\bb{B}}$ since $\Vdash_{\bb{B}} \name{A}_i \in \cal{J}^+$. Let $X \subseteq Y \setminus \alpha_{\star}$ be such that $\otp(X) = \omega_1$ and for every $\alpha < \beta$ in $X$, $f(\alpha) < f(\beta)$.

\begin{claim}
\label{c217}
For every $i < \theta$, $\{p_{i, \alpha}: \alpha \in X\}$ is predense in $\bb{B}$.
\end{claim}

\begin{proof} Let $\sup(X) = \gamma_{\star}$. Then $\cf(\gamma_{\star}) = \aleph_1$ and hence $\bb{B}_{\gamma_{\star}} = \bigcup \{\bb{B}_{\gamma}: \gamma \in X\}$. Fix $i < \theta$. Given $p \in \bb{B}$, choose $\gamma \in X$ such that $\pi_{\gamma_{\star}}(p) \in \bb{B}_{\gamma}$. Now since

$$\bb{B} = \bb{B}_{\gamma} \star \bb{B}_{\gamma_{\star}} \slash \bb{B}_{\gamma} \star \bb{B} \slash \bb{B}_{\gamma_{\star}}$$

we can decompose $p = (\pi_{\gamma_{\star}}(p), 1, x)$ and $p_{i, \gamma} = (1, y, 1)$. Hence $p, p_{i, \gamma}$ are compatible.  \end{proof}

It follows that $\cal{J}$ is $\omega_2$-supersaturated. \end{proof}

\section{Consistently, there are $\omega_1$-saturated ideals on $\cont$ and all of them are supersaturated}

The aim of this section is to show that it is consistent that every $\omega_1$-saturated $\sigma$-ideal is supersaturated.

\begin{thm}
\label{t31}
It is consistent that there is a normal supersaturated ideal on $\cont$ and every $\omega_1$-saturated $\sigma$-ideal is supersaturated.
\end{thm}

\begin{lem}
\label{l32}
Suppose that every $\cal{I}$ satisfying (i)-(iv) below is supersaturated.

\begin{itemize}
\item[(i)] $\cal{I}$ is a uniform ideal on $\lambda$,

\item[(ii)] $\omega_1 \leq \mu \leq \lambda$,

\item[(iii)] for every $X \in \cal{I}^+$, $\add(\cal{I} \res X) = \mu$ and 

\item[(iv)] $\cal{I}$ is $\omega_1$-saturated.  

\end{itemize}

Then every $\omega_1$-saturated $\sigma$-ideal is supersaturated.
\end{lem}

\begin{proof}
Suppose $\cal{J}$ is an $\omega_1$-saturated $\sigma$-ideal on $X$. Note that for every $A \in \cal{J}^+$, there exists $B \subseteq A$ such that $(\star)_B$ holds where \\

$(\star)_B$ says the following:  $B \in \cal{J}^+$, $[B]^{< |B|} \subseteq \cal{J}$ and for every $C \subseteq B$, if $C \in \cal{J}^+$, then $\add(\cal{J} \res C) = \add(\cal{J} \res B)$. \\

Since $\cal{J}$ is $\omega_1$-saturated, we can find a countable partition $\cal{F}$ of $X$ such that for each $B \in \cal{F}$, $(\star)_B$ holds. Now by assumption, each $\cal{J} \res B$ is supersaturated. Hence $\cal{J}$ is also supersaturated. \end{proof}

\begin{lem}
\label{l33}
Suppose $\bb{P}$ is a ccc forcing, $\kappa > \cont$ and $V^{\bb{P}} \models \cal{J}$ is a $\kappa$-complete $\omega_1$-saturated uniform ideal on $\lambda$. Let $\cal{I} = \{X \subseteq \kappa: 1_{\bb{P}} \Vdash X \in \cal{J} \}$. Then there is a countable partition $\cal{F}$ of $\lambda$ such that for every $A \in \cal{F}$, $\cal{I} \res A = \{Y \subseteq \lambda: Y \cap A \in \cal{I}\}$ is a $\kappa$-complete prime ideal on $\lambda$. 
\end{lem}

\begin{proof}
It is clear that $\cal{I}$ is a $\kappa$-complete uniform ideal on $\lambda$. Suppose $\cal{F} \subseteq \cal{I}^+$ is an uncountable family of pairwise disjoint sets. For each $A \in \cal{F}$, choose $p_A \in \bb{P}$ such that $p_A \Vdash A \notin \cal{J}$. Since $\bb{P}$ is ccc, some $p \in \bb{P}$ forces uncountably many $p_A$'s into the $\bb{P}$-generic filter. But this contradicts the fact that $\cal{J}$ is $\omega_1$-saturated in $V^{\bb{P}}$. So $\cal{I}$ is $\omega_1$-saturated. Since $\cal{I}$ is $\kappa$-complete and $\kappa > \cont$, $\cal{I}$ is nowhere atomless. Hence there is a countable partition $\cal{F}$ of $\lambda$ such that for every $A \in \cal{F}$, $\cal{I} \res A = \{Y \subseteq \lambda: Y \cap A \in \cal{I}\}$ is a $\kappa$-complete prime ideal on $\lambda$.
\end{proof}

\begin{lem}
\label{l34}
Suppose $\kappa$ is an inaccessible cardinal and $\cal{U}$ is a $\kappa$-complete uniform ultrafilter on $\lambda$. Let $\bb{P} = \Cohen_{\kappa}$. Let $\cal{J}$ be the ideal generated by the dual ideal of $\cal{U}$ in $V^{\bb{P}}$. Then for each $\cal{A} \subseteq \cal{J}^+$, if $|\cal{A}| < \kappa$, then there exists a countable set that meets every member of $\cal{A}$.
\end{lem}

\begin{proof}
It is clear that $\cal{J}$ is a $\kappa$-complete uniform ideal on $\lambda$. Suppose $\theta < \kappa$ and $\langle \name{A}_i: i < \theta \rangle$ is a sequence of $\cal{J}$-positive sets in $V^{\bb{P}}$. WLOG, assume that the empty condition forces this. For $i < \theta$ and $\alpha < \lambda$, let $p_{i, \alpha} = [[\alpha \in \name{A}_i]]_{\bb{P}}$. Note that for each $i < \theta$, and $Z \in \cal{U}$, $\{p_{i, \alpha}: \alpha \in Z\}$ is predense in $\bb{P}$ since otherwise some condition will force $\name{A}_i \in \cal{J}$. Since $\cal{U}$ is $\kappa$-complete, we can choose $X \in \cal{U}$ such that for every $i < \theta$ and $\alpha \in X$, $p_{i, \alpha} > 0_{\bb{P}}$. Let $S_{i, \alpha} \in [\kappa]^{\aleph_0}$ be a support of $p_{i, \alpha}$. Since $\kappa$ is inaccessible, we can choose $Y \subseteq X$ such that $Y \in \cal{U}$ and for each $i < \theta$, the following hold.

\begin{itemize}

\item[(a)] For every $\alpha, \beta \in Y$, $(S_{i, \alpha}, 2^{S_{i, \alpha}}, p_{i, \alpha}) \cong (S_{i, \beta}, 2^{S_{i, \beta}}, p_{i, \beta})$. Put $\otp(S_{i, \alpha}) = \gamma_i$. Let $h_{i, \alpha}:\gamma_i \to S_{i, \alpha}$ be the order isomorphism and define $H_{i, \alpha}: 2^{\gamma_i} \to 2^{S_{i, \alpha}}$ by $H_{i, \alpha}(x) = x \circ h^{-1}_{i, \alpha}$. Choose $p_i \subseteq 2^{\gamma_i}$ such that $H_{i, \alpha}[p_i] = p_{i, \alpha}$.

\item[(b)] For each $\gamma < \gamma_i$, either $|\{h_{i, \alpha}(\gamma): \alpha \in Y\}| = 1$ or for every $Z \in \cal{U}$, $|\{h_{i, \alpha}(\gamma): \alpha \in Z \cap Y\}| \geq \kappa$. Put $\Gamma_i = \{\gamma < \gamma_i: |\{h_{i, \alpha}(\gamma): \alpha \in Y\}| = 1 \} $ and $h_{i, \alpha}[\Gamma_i] = R_i$. 

\end{itemize}

Define $$B_{i, \alpha} = \{x \in 2^{R_i}: \{y \res (S_{i, \alpha} \setminus R_i): y \in p_{i, \alpha} \wedge y \res R_i = x \} \text{ is meager}\}.$$

Then $B_{i, \alpha} = B_i$ does not depend on $\alpha \in Y$ and $B_i$ is meager in $2^{R_i}$ since otherwise $\{p_{i, \alpha}: \alpha \in Y\}$ will not be predense in $\bb{P}$. \\

Using (b), choose $B \in [Y]^{\aleph_0}$ such that for every $i < \theta$ and $\alpha \neq \beta$ in $B$, $S_{i, \alpha} \cap S_{i, \beta} = R_i$. It follows now that for every $i < \theta$, $\{p_{i, \alpha}: \alpha \in B\}$ is predense in $\bb{P}$. Hence $\Vdash (\forall i < \theta)(B \cap \name{A}_i \neq \emptyset)$. \end{proof}

\textbf{Proof of Theorem \ref{t31}}: Let $V \models ``\cont = \omega_1$ and $\kappa$ is the least measurable cardinal". Let $\bb{P} = \Cohen_{\kappa}$. We already know that there is a normal supersaturated ideal on $\kappa = \cont$ in $V^{\bb{P}}$. Let us check that, $V^{\bb{P}} \models $ ``Every $\omega_1$-saturated $\sigma$-ideal is supersaturated". By Lemma \ref{l32}, it suffices to consider ideals $\cal{J}$ that satisfy the following for some $\omega_1 \leq \mu \leq \lambda$.

\begin{itemize}

\item[(i)] $\cal{J}$ is a uniform ideal on $\lambda$, 

\item[(ii)] for every $X \in \cal{J}^+$, $\add(\cal{J} \res X) = \mu$ and 

\item[(ii)] $\cal{J}$ is $\omega_1$-saturated.  

\end{itemize}

Since $V^{\bb{P}} \models \cont = \kappa$, we can assume that $\mu \leq \kappa$. Otherwise there is a countable partition $\cal{E}$ of $\lambda$ into $\cal{J}$-positive sets such that for each $X \in \cal{E}$, $\cal{J} \res X$ is a $\mu$-complete prime ideal and it easily follows that $\cal{J}$ is supersaturated. \\

Towards a contradiction, suppose $\mu < \kappa$. Then, in $V^{\bb{P}}$, there is a $\mu$-additive $\omega_1$-saturated ideal $\cal{K}$ on $\mu$. Put $\cal{K}' = \{X \subseteq \mu: 1_{\bb{P}} \Vdash X \in \cal{K}\}$. Since $\bb{P}$ is ccc, $V \models \cal{K}$ is a $\mu$-additive $\omega_1$-saturated ideal on $\mu$. As $V \models \mu > \omega_1 = \cont$, $\mu$ is measurable in $V$.  But $\kappa$ is the least measurable cardinal in $V$. Hence $\mu \geq \kappa$: Contradiction. \\

So we must have $\mu = \kappa$. Let $\cal{I} = \{Y \subseteq \lambda: 1_{\bb{P}} \Vdash X \in \cal{J}\}$. By Lemma \ref{l33}, there is a countable partition $\cal{F}$ of $\lambda$ such that for each $X \in \cal{F}$, $\cal{I} \res X$ is a $\kappa$-complete prime ideal on $\lambda$. For each $X \in \cal{F}$, let $\cal{I}_X$ be the ideal generated by $\cal{I} \res X$ in $V^{\bb{P}}$. By Lemma \ref{l34}, for every $\cal{A} \subseteq \cal{I}_X^+$, if $|\cal{A}| < \kappa$, then there is a countable set that meets every member of $\cal{A}$. Since $\cal{I}_A \subseteq \cal{J} \res A$ and $\add(\cal{J} \res A) = \kappa$, it follows that $\cal{J} \res A$ is supersaturated for each $A \in \cal{F}$. Since $\cal{F}$ is a countable, partition of $\lambda$, it follows that $\cal{J}$ is also supersaturated. \qed

\section{Killing supersaturated ideals}
\label{s4}

\begin{defn}
\label{d41}
Suppose $\delta < \omega_1$ is indecomposable and $\kappa$ is an infinite cardinal. Let $\bb{Q}^{\kappa}_{\delta}$ consist of all countable partial maps from $\kappa$ to $2$ such that

\begin{itemize}

\item[(1)] $\otp(\dom(p)) < \delta$ and

\item[(2)] $\{\xi \in \dom(p): p(\xi) = 1\}$ is finite. 

\end{itemize} 
For $p, q \in \bb{Q}^{\kappa}_{\delta}$ define $p \leq q$ iff $q \subseteq p$. Let $\bb{P}_{\kappa}$ be the finite support product of $\{Q^{\kappa}_{\delta}: \delta < \omega_1, \delta \text{ indecomposable}\}$.
\end{defn}

\begin{lem}
\label{l42} Let $\bb{P}_{\kappa}$ be as in Definition \ref{d41}.
\begin{itemize}
\item[(1)] $\bb{P}_{\kappa}$ is ccc.
\item[(2)] If $\kappa \geq \omega_1$, then $\bb{P}_{\kappa}$ is not $\sigma$-finite-cc.
\end{itemize}

\end{lem}

\begin{proof}

(1) Towards a contradiction, suppose $A = \{p_i: i < \omega_1\}$ is an uncountable antichain in $\bb{P}_{\kappa}$. Put $D_i = \dom(p_i)$. By passing to an uncountable subset of $A$, we can assume that $D_i$'s form a $\Delta$-system with root $D$. For each $\delta \in D$ and $i < \omega_1$, put $s_{i, \delta} = \{\gamma: p_i(\delta)(\gamma) = 1\}$ and $X_{i, \delta} = \{\gamma: p_i(\delta)(\gamma) = 0\}$. Note that $\otp(X_{i, \delta}) < \delta$. Choose $B \in [A]^{\omega_1}$ such that for each $\delta \in D$, $\langle s_{i, \delta}: i \in B \rangle$ is a $\Delta$-system with root $s_{\delta}$ and for every $i < j$ in $B$, $\displaystyle s_{j, \delta} \cap X_{i, \delta} = \emptyset$. 

Choose $j \in B$ and $\delta \in D$ such that letting $C = \{i \in B \cap j: p_i(\delta) \perp_{\bb{Q}_{\delta}} p_j(\delta) \}$, every transversal of $\{s_{i, \delta} \setminus s_{\delta}: i \in C\}$ has order type $\geq \delta$. Now observe that $X_{j, \delta}$ has to meet $s_{i, \delta} \setminus s_{\delta}$ for every $i \in C$. Hence $ \otp(X_{j, \delta}) \geq \delta$: Contradiction.  \\

(2) It is enough to show that $\bb{Q} = \bb{Q}^{\omega_1}_{\omega^2}$ is not $\sigma$-finite-cc. Towards a contradiction, suppose $\bb{Q} = \bigsqcup_{n < \omega} W_n$ where no $W_n$ has an infinite antichain.  Choose $\langle A_n : n < \omega \rangle$ as follows.

\begin{itemize}
\item[(a)] $A_0 \subseteq W_0$ is a maximal antichain of conditions $p$ such that $\max(\dom(p)) = \gamma_p$ exists and $p(\gamma_p) = 1$.  Define $\gamma_0 = \max(\{\gamma_p: p \in A_0\})$.

\item[(b)] $A_{n+1} \subseteq W_{n+1}$ is a maximal antichain of conditions $p \in W_{n+1}$ such that $\max(\dom(p)) = \gamma_p$ exists, $\gamma_p > \gamma_n$ and $p(\gamma_p) = 1$.  If $A_{n+1} \neq \emptyset$,  define $\gamma_{n+1} = \max(\{\gamma_p: p \in A_{n+1}\})$.  Otherwise,  $\gamma_{n+1} = \gamma_n$.

\end{itemize}

Put $A = \bigcup_{n < \omega} A_n$ and $\gamma = \sup(\{\gamma_n: n < \omega\})$.  Fix $\gamma_{\star} \in (\gamma,  \omega_1)$.  Let $p_{\star}$ be defined by $\dom(p_{\star}) = \{\gamma_p: p \in A\} \cup \{\gamma_{\star}\}$ and for every $\xi \in \dom(p_{\star})$,  $p(\xi) = 1$ iff $\xi = \gamma_{\star}$.  Note that $\otp(\dom(p)) \leq \omega + 1 < \omega^2$ and hence $p_{\star} \in \bb{Q}$. Choose $n < \omega$ such that $p_{\star} \in W_n$.  But now $A_n \cup \{p_{\star}\} \subseteq W_n$ is an antichain which contradicts the maximality of $A_n$.\end{proof}

\begin{thm}
\label{t43}
Suppose $\omega_1 \leq \kappa \leq \lambda$, $\cal{I}$ is an $\omega_1$-saturated uniform ideal on $\lambda$ and $\add(\cal{I}) = \kappa$. Let $\bb{P}_{\kappa}$ be as in Definition \ref{d41}. Let $\cal{J}$ be the ideal generated by $\cal{I}$ in $V^{\bb{P}_{\kappa}}$. Then there exists $\cal{A} \subseteq \cal{J}^+$ such that $|\cal{A}| = \omega_1$ and there is no countable set that meets every member of $\cal{A}$. Hence $V^{\bb{P}_{\kappa}} \models \cal{J}$ is an $\omega_1$-saturated $\kappa$-complete uniform ideal on $\lambda$ which is not supersaturated.
\end{thm}

\begin{proof}

As $\bb{P}_{\kappa}$ is ccc, it is easy to see that in $V^{\bb{P}_{\kappa}}$, $\cal{J}$ is an $\omega_1$-saturated $\kappa$-complete uniform ideal on $\lambda$. So it suffices to show that in $V^{\bb{P}_{\kappa}}$, there exists $\cal{A} \subseteq \cal{J}^+$ such that $|\cal{A}| = \omega_1$ and there is no countable set that meets every member of $\cal{A}$.\\

Since $\add(\cal{I}) = \kappa$, we can fix $Y \in \cal{I}^+$ and a partition $Y = \bigsqcup_{\alpha < \kappa} W_{\alpha}$ such that for each $\Gamma \in [\kappa]^{< \kappa}$, $\bigcup_{\alpha \in \Gamma} W_{\alpha} \in \cal{I}$. Let $G$ be $\bb{P}_{\kappa}$-generic over $V$. Let $G_{\delta} = \{p(\delta): p \in G\}$. So $G_{\delta}$ is $\bb{Q}_{\delta}$-generic over $V$. Define $\name{A}_{\delta} \in V^{\bb{P}_{\kappa}} \cap \cal{P}(\lambda)$ by 

$$\gamma \in \name{A}_{\delta} \iff (\exists p \in G)(p(\delta)(\alpha) = 1 \wedge \gamma \in W_{\alpha})$$   \\

Suppose $Y \in \cal{I}$ and $p \in \bb{P}_{\kappa}$ with $\delta \in \dom(p)$. Choose $\alpha < \kappa$ such that $W_{\alpha} \setminus Y \neq \emptyset$ and $\alpha \notin \dom(p(\delta))$. Let $q \leq p$ be such that $q(\delta)(\alpha) = 1$. Then $q \Vdash_{\bb{P}_{\kappa}} \name{A}_{\delta} \setminus Y \neq \emptyset$. Hence $\Vdash_{\bb{P}_{\kappa}} \name{A}_{\delta} \in \cal{J}^+$.  \\

Towards a contradiction suppose that in $V^{\bb{P}_{\kappa}}$, there is a countable $X \subseteq \lambda$ that meets each $\name{A}_{\delta}$. Since $\bb{P}$ satisfies ccc, we can assume that $X \in V$. Fix $p \in \bb{P}_{\kappa}$ such that $p \Vdash_{\bb{P}} (\forall \delta)(X \cap \name{A}_{\delta} \neq \emptyset)$. Put $W = \{\alpha < \kappa: W_{\alpha} \cap X \neq \emptyset\}$. So $W \subseteq \kappa$ is countable. Choose $\delta \in \omega_1 \setminus \dom(p)$ indecomposable such that $\delta > \otp(W)$. Define $q \in \bb{P}_{\kappa}$ by $\dom(q) = \dom(p) \cup \{\delta\}$, $q \res \dom(p) = p$ and $q(\delta) \in \bb{Q}_{\delta}$ is constantly zero on $W$. Then $q \leq p$ and $q \Vdash_{\bb{P}_{\kappa}} X \cap \name{A}_{\delta} = \emptyset$: Contradiction. It follows that $\cal{A} = \{A_{\delta}: \delta < \omega_1, \delta \text{ indecomposable}\}$ is as required.
\end{proof}

\begin{defn}
\label{d44}
Let $\langle (\bb{S}_i, \bb{R}_j): i \leq \kappa^+, j < \kappa \rangle$ be the finite support iteration defined by

\begin{itemize}

\item[(a)] $\bb{S}_0$ is the trivial forcing.

\item[(b)] For each $i < \kappa^+$, $V^{\bb{S}_i} \models \bb{R}_i = \bb{P}_{\kappa}$.  
 
\end{itemize}

\end{defn}

The next theorem shows how to kill all atomless supersaturated ideals.

\begin{thm}
\label{t45}
Suppose $V \models ``\cont = \omega_1$ and $\kappa$ is the least measurable cardinal with a witnessing normal prime ideal $\cal{I}$".  Put $\bb{S} = \bb{S}_{\kappa^+}$. Then the following hold in $V^{\bb{S}}$. 

\begin{itemize}

\item[(a)] $\cont = \kappa^+$ and the ideal generated by $\cal{I}$ is a normal $\omega_1$-saturated ideal on $\kappa$.
\item[(b)] Whenever $\cal{J}$ is a supersaturated ideal on a set $X$, there is a countable partition $\cal{F}$ of $X$ such that for each $A \in \cal{F}$, $\cal{J} \res A$ is a prime ideal. In particular, there is no supersaturated ideal on any cardinal $\leq \cont$.

\end{itemize}

\end{thm}

\begin{fact}
\label{f44}
Suppose $\cal{I}_1$, $\cal{I}_2$ are $\omega_1$-saturated $\sigma$-ideals on $X$ and $\cal{I}_1 \subseteq \cal{I}_2$. Then there is a partition $X = A \sqcup B$ such that $A \in \cal{I}_2$ and $\cal{I}_2 \res B = \cal{I}_1 \res B$.
\end{fact}

\begin{proof}
Take $A$ to be the union of a maximal family of pairwise disjoint sets in $\cal{I}_2 \setminus \cal{I}_1$.
\end{proof}

\textbf{Proof of Theorem} \ref{t45}: Clause (a) is easy to check. Let us prove Clause (b). Suppose $\cal{J}$ is a supersaturated ideal on $X$. Put $\mu = \add(\cal{J})$. It suffices to show that $V^{\bb{S}} \models \mu > \cont$ since this would imply that $\cal{J}$ is nowhere atomless and hence there is a countable partition of $X$ into $\cal{J}$-positive sets such that the restriction of $\cal{J}$ to each of them is a prime ideal. Towards a contradiction, assume $V^{\bb{S}} \models \mu \leq \cont$. Fix $Y \in \cal{J}^+$ and a partition $Y = \bigsqcup_{\alpha < \mu} W_{\alpha}$ such that for every $\Gamma \in [\mu]^{< \mu}$, $\bigcup_{\alpha \in \Gamma} W_{\alpha} \in \cal{J}$. Define $$\cal{K} = \{\Gamma \subseteq \mu: \bigcup_{\alpha \in \Gamma} W_{\alpha} \in \cal{J}\}$$ 

Then $\cal{K}$ is a $\mu$-additive ideal on $\mu$. We claim that $\cal{K}$ must also be supersaturated. To see this, suppose $\cal{A} \subseteq \cal{K}^+$ and $|\cal{A}| < \mu$. For each $A \in \cal{A}$, define $Y_A = \bigsqcup_{\alpha \in A} W_{\alpha}$. Then $\{Y_A: A \in \cal{A}\} \subseteq \cal{J}^+$. Since $\cal{J}$ is supersaturated, we can choose a countable $T \subseteq Y$ that meets each $Y_A$. Let $B = \{\alpha < \mu: Y \cap W_{\alpha} \neq \emptyset\}$. Then $B \subseteq \mu$ is countable and it meets every $A \in \cal{A}$. Hence $\cal{K}$ is supersaturated.  WLOG, assume that the empty condition in $\bb{S}$ forces all of this about $\cal{K}$.  \\

Since $V^{\bb{S}} \models ``\mu \leq \cont = \kappa^+$ and $\mu$ is weakly inaccessible", we must have $\mu \leq \kappa$.  We consider two cases. \\

Case $\mu < \kappa$: In $V$, define $\cal{I}' = \{X \subseteq \mu: 1_{\bb{S}} \Vdash X \in \cal{K}\}$. Since $\bb{S}$ is ccc, $V \models \cal{I}'$ is a $\mu$-additive $\omega_1$-saturated ideal on $\mu$. As $V \models \mu > \omega_1 = \cont$, $\mu$ is measurable in $V$. Since $\kappa$ is the least measurable cardinal in $V$, $\mu \geq \kappa$: Contradiction. \\

Case $\mu = \kappa$: In $V$, define $\cal{I}' = \{X \subseteq \kappa: 1_{\bb{S}} \Vdash X \in \cal{K}\}$. Since $V \models \kappa > \cont = \omega_1$, we must have $V \models \cal{I}'$ is a $\kappa$-additive prime ideal on $\kappa$. Let $\cal{K}'$ be the ideal generated by $\cal{I}'$ in $V^{\bb{S}}$. Then $V^{\bb{S}} \models \cal{K}' \subseteq \cal{K}$ are $\omega_1$-saturated $\kappa$-additive ideals on $\kappa$. Using Fact \ref{f44}, fix $B \in \cal{K}^+$ such that $\cal{K}' \res B = \cal{K} \res B$. 

Choose $\gamma < \kappa^+$ such that $\name{B} \in V^{\bb{S}_{\gamma}}$. Let $\cal{K}''$ be the ideal generated by $\cal{I}'$ in $V^{\bb{S}_{\gamma}}$. By Theorem \ref{t43}, it follows that in $V^{\bb{S}_{\gamma + 1}}$, the ideal generated by $\cal{K}'' \res B$ is not supersaturated. Now observe that $\cal{K} \res B = \cal{K}' \res B$ is the ideal generated by $\cal{K}'' \res B$ in $V^{\bb{S}}$. It follows that $\cal{K}$ is not a supersaturated ideal: Contradiction.  \qed \\

Using some results about separating families and supersaturated ideals from \cite{HLRS, KR}, we can also get the following.

\begin{thm}
Suppose $\kappa$ is a measurable cardinal with a witnessing normal prime ideal $\cal{I}$. Let $\bb{P}_{\kappa}$ be the forcing in Definition \ref{d41}. Then the following hold in $V^{\bb{P}_{\kappa}}$.

\begin{itemize}

\item[(a)] $\cont = \kappa$ and the ideal generated by $\cal{I}$ is a normal $\omega_1$-saturated ideal on $\kappa$.

\item[(b)] There is a family $\cal{F} \subseteq \cal{P}(\kappa)$ such that $|\cal{F}| = \omega_1$ and for every countable $X \subseteq \kappa$ and $\alpha \in \kappa \setminus X$, there exists $S \in \cal{F}$ such that $\alpha \in S$ and $S \cap X = \emptyset$.

\item[(c)] The order dimension of Turing degrees is $\omega_1$.

\item[(d)] There are no atomless supersaturated ideals.

\end{itemize}

\end{thm}

\begin{proof}

(a) Since $\bb{Q}^{\kappa}_{\omega}$ adds $\kappa$ Cohen reals, $\cont \geq \kappa$. The other inequality follows by a name counting argument using the fact that $|\bb{P}_{\kappa}| = \kappa$. That the ideal generated by $\cal{I}$ is a normal $\omega_1$-saturated ideal on $\kappa$ follows from the fact that $\bb{P}_{\kappa}$ is ccc. \\

(b) For each indecomposable $\delta < \omega_1$, define $$S_{\delta} = \{\alpha < \kappa: (\exists p \in G_{\bb{P}_{\kappa}})(\delta \in \dom(p) \wedge p(\delta)(\alpha) = 1)\}$$

Let $\cal{F} = \{S_{\delta}: \delta < \omega_1 \text{ is indecomposable}\}$. Suppose $X \subseteq \kappa$ is countable and $\alpha \in \kappa \setminus X$. We'll find an $S_{\delta} \in \cal{F}$ such that $\alpha \in S_{\delta}$ and $X \cap S_{\delta} = \emptyset$. Since $\bb{P}_{\kappa}$ is ccc, we can find a countable $Y \in V$ such that $X \subseteq Y \subseteq \kappa \setminus \{\alpha\}$. Now an easy density argument shows that the set $$D_{\alpha, Y} = \{p \in \bb{P}_{\kappa}:(\exists \delta \in \dom(p))[p(\delta)(\alpha) = 1 \wedge (\forall \beta \in Y)(p(\delta)(\beta) = 0)]\}$$

is dense in $\bb{P}_{\kappa}$. So we can choose $p \in D_{\alpha, Y} \cap G_{\bb{P}_{\kappa}}$. Let $\delta$ witness that $p \in D_{\alpha, Y}$. Then it is clear that $\alpha \in S_{\delta}$ and $X \cap S_{\delta} \subseteq Y \cap S_{\delta} = \emptyset$. \\

(c) This follows from Theorem 3.9 in \cite{HLRS} and part (b) above. \\

(d) Suppose not. Then arguing as in first two paragraphs of the proof of Theorem \ref{t45} above, we can find some $\mu \leq \cont = \kappa$ and a $\mu$-additive supersaturated ideal on $\mu$. Define $\cal{E} = \{S \cap \mu: S \in \cal{F}\}$ where $\cal{F}$ is as in part (b). Then $|\cal{E}| = \omega_1$ and for every countable $X \subseteq \mu$ and $\alpha \in \mu \setminus X$, there exists $S \in \cal{E}$ such that $\alpha \in S$ and $S \cap X = \emptyset$. Now applying Lemma 4.2 in \cite{KR} gives us a contradiction.
\end{proof}

We conclude with the following questions.

\begin{itemize}

\item[(1)] Suppose $\cal{I}, \cal{J}$ are normal ideals on $\kappa$, $\cal{I}$ is supersaturated and $\cal{P}(\kappa) \slash \cal{I}$ is isomorphic to $\cal{P}(\kappa) \slash \cal{J}$. Must $\cal{J}$ be supersaturated?

\item[(2)] Suppose $\kappa$ is regular uncountable, $\cal{I}$ is a $\kappa$-complete normal ideal on $\kappa$ and $\cal{P}(\kappa) \slash \cal{I}$ is a Cohen algebra. Must $\cal{I}$ be supersaturated?

\item[(3)] Do $\sigma$-finite/bounded-cc forcings preserves supersaturation? What about Boolean algebras that admit a strictly positive finitely additive measure?

\end{itemize}

\end{document}